\documentclass[reqno, 12pt]{amsart}

\usepackage{array}
\usepackage{amsmath}
\usepackage{amsfonts}
\usepackage{amssymb}
\usepackage{enumerate}
\usepackage{amsthm}
\usepackage{amsmath, amscd}
\usepackage{caption}
\usepackage[usenames, dvipsnames]{color}
\usepackage{xy}
\xyoption{all}
\usepackage{hyperref}
\usepackage{tikz}
\usepackage{tikz-cd}
\usepackage{tikz-3dplot}
\tdplotsetmaincoords{70}{120}
\usepackage{marginnote}
\usetikzlibrary{matrix,arrows,backgrounds}
\newtheorem{theorem}{Theorem}[section]
\newtheorem{lemma}[theorem]{Lemma}
\newtheorem{proposition}[theorem]{Proposition}
\newtheorem{corollary}[theorem]{Corollary}
\newtheorem{conjecture}[theorem]{Conjecture}

\theoremstyle{definition}
\newtheorem{definition}[theorem]{Definition}

\newtheorem{remark}[theorem]{Remark}

\newtheorem*{theorem*}{Theorem}
\newtheorem*{corollary*}{Corollary}
\newtheorem*{conjecture*}{Conjecture}
\newtheorem*{question*}{Question}
\numberwithin{equation}{section}

\newcommand{\newterm}{\textsf}

\newcommand{\dbcoh}[1]{\operatorname{D}^{\operatorname{b}}(\operatorname{coh }#1)}
\newcommand{\dbmod}[1]{\operatorname{D}^{\operatorname{b}}(\operatorname{mod }#1)}

\newcommand{\Hom}{\operatorname{Hom}}

\newcommand{\Z}{\mathbb{Z}}
\newcommand{\R}{\mathbb{R}}
\newcommand{\C}{\mathbb{C}}
\newcommand{\Q}{\mathbb{Q}}
\newcommand{\A}{\mathbb{A}}
\def\O{\mathcal{O}}

\newcommand{\End}{\operatorname{End}}
\newcommand{\spec}{\operatorname{Spec}}
\newcommand{\Rhom}{\operatorname{RHom}}

\newcommand{\Ext}{\operatorname{Ext}}
\newcommand{\gldim}{\operatorname{gl}\dim}
\newcommand{\tot}{\operatorname{tot}}
\newcommand{\cone}{\operatorname{Cone}}

\newcommand{\conv}{\operatorname{Conv}}

\title[NCCRs for (almost) simplicial toric algebras]{Non-commutative crepant resolutions for (almost) simplicial toric algebras}

\author{Aimeric Malter \textsuperscript{1}}
\email{aimericmalter@bimsa.cn}
\author{Artan Sheshmani \textsuperscript{1,2}}
\email{artan@mit.edu}
\address{\textsuperscript{1}Beijing Institute of Mathematical Sciences and Applications, No. 544, Hefangkou Village, Huaibei Town, Huairou District, Beijing 101408}
\address{\textsuperscript{2}Massachusetts Institute of Technology, IAiFi Institute, 77 Massachusetts Ave, 26-555. Cambridge, MA 02139}

\begin{document}

\begin{abstract}
Given a rational convex polyhedral Gorenstein cone constructed as cone over a lattice polytope $P$, we establish that toric non-commutative crepant resolutions (NCCRs) of its associated toric algebra descend to toric NCCRs of the algebras associated to faces of the polytope $P$. As consequence, we present two new, short proofs to the existence of toric NCCRs for simplicial affine toric Gorenstein algebras and for almost simplicial affine toric Gorenstein algebras, i.e. those associated to cones $\sigma$ with $\dim\sigma+1$ extremal rays. 
\end{abstract}
\maketitle
\section{Introduction}\label{Sec:Intro}

Resolutions of singularities have long been an intense area of study in algebraic geometry. Oftentimes, the underlying geometric objects become more complicated after resolving them (for instance, most singular affine schemes do not have affine resolutions). The classical approach to deal with this fact of life is to simply expand the class of geometric objects one considers. A second approach that emerged is to abstracise and distill from the resolutions some algebraic properties which in turn can be seen as defining characteristics of resolutions. For instance, Bondal-Orlov \cite{BO02} and Kawamata \cite{Kaw02} conjectured that given a normal algebraic variety $X$ with two crepant resolutions $\pi_i:Y_i\rightarrow X$, $i=1,2$, there is an equivalence of triangulated categories $\dbcoh{Y_1}\cong \dbcoh{Y_2}$ linear over $X$.

Verifying this conjecture for a simple flop of three-dimensional varieties, Van den Bergh \cite{VdB3Dflops} used tilting theory to construct two Morita equivalent algebras $\Lambda_i$ such that $\dbmod{\Lambda_i}\cong \dbcoh{Y_i}$. In a certain sense thus, the algebras $\Lambda_i$ themselves constitute crepant resolutions. Van den Bergh \cite{VdB04} formalises this notion, introducing \newterm{non-commutative (crepant) resolutions}, abbreviated as NC(C)Rs. For a normal noetherian domain $R$, an NCR is an algebra of the form $\Lambda=\End_R(M)$ for $M\in \operatorname{ref} R$ such that $\gldim\Lambda<\infty$. If further $R$ is Gorenstein, we call $\Lambda$ crepant if it is a maximal Cohen-Macaulay $R$-module. In general, the existence of such resolutions is thus far an open question even for the case of affine toric Gorenstein algebras, i.e. the algebras associated to strictly convex rational polyhedral Gorenstein cones. However, certain classes of cones have been shown to possess NCCRs, using a plethora a methods of algebraic and combinatorial nature (see, e.g.,\cite{Broomhead, svdbt2, FMS19,Tomo25}). \\
All Gorenstein cones $\sigma$ can be written (after change of basis) as $\cone(P\times\{1\})$, where $P$ is a lattice polytope. In \cite{MS25}, we conjecture that for a lattice polytope $Q$ lattice equivalent to a face of $P$, if the coordinate ring of $\sigma_P=\cone(P\times \{1\})$ admits an NCCR, so does the coordinate ring of $\sigma_Q=\cone(Q\times\{1\})$. The main result of the present paper provides a proof of this conjecture for the special case of \newterm{toric} NCCR, which arise when the rank of the reflexive $R$-module $M$ is 1.

\begin{theorem*}(Theorem \ref{Thm:FaceOfReflexiveTORICNCCR})
    Let $Q\subset \R^k$ be a lattice polytope and let $\sigma$ be the cone $\sigma=\cone(Q\times\{1\})\subset \R^{k+1}$ together with its toric algebra $R_\sigma=k[\sigma^\vee\cap M_\sigma]$, where $M_\sigma$ is the character lattice of the toric variety defined by $\sigma$. Suppose $Q$ is lattice equivalent to a face $F$ of a lattice polytope $P\subset \R^n$. Consider the cone $\sigma'=\cone(P\times\{1\})\subset\R^{n+1}$. Denoting the corresponding character lattice by $M'$, let $R'=k[(\sigma')^\vee\cap M']$ be the toric algebra associated to $\sigma'$. Suppose $R'$ has a toric NCCR $\Lambda'$. Then $R$ has a toric NCCR.
\end{theorem*}

With this theorem, we develop a unified approach to providing toric NCCRs for Gorenstein toric algebras, by expressing a Gorenstein cone $\sigma$ as a facet of another Gorenstein cone $\sigma'$ and descending the corresponding NCCR. 
In this way, we obtain a new, short proof of the known result that all Gorenstein simplicial toric algebras admit an NCCR.

 \begin{corollary*}(Corollary \ref{Cor:SimplicialHasNCCR})
    Let $\sigma\subset N_{\R}$ be a simplicial Gorenstein cone. Then $R=k[\sigma^\vee\cap M]$ has an NCCR.
\end{corollary*}

Furthermore, we can use our approach prove the existence of toric NCCRs for the case of \newterm{almost simplicial} Gorenstein cones (i.e. Gorenstein cones $\sigma$ with $\dim \sigma+1$ extremal rays), remarking that this result has recently been proven by Tomonaga \cite{Tomo25} using different, more algebraic methods.

\begin{theorem*}(Theorem \ref{Thm:AlmSim})
    Let $\sigma\subset N_\R$ be an almost simplicial Gorenstein cone, i.e. $\sigma=\cone(P\times\{1\})$ with $P$ a lattice polytope with $\dim P+2$ vertices. Then $R=k[\sigma^\vee\cap M]$ has a toric NCCR.
\end{theorem*}

\noindent Our proof works by constructing, given a lattice polytope $P$, another polytope $Q$ in one dimension higher such that a simplicial subdivision of the associated face fan yields a smooth toric DM stack with a tilting bundle (obtained by applying methods of Borisov-Hua \cite{BH09}). This tilting bundle further fulfills some cohomological vanishing conditions which, by results from our previous work \cite{MS25}, guarantee the existence of a toric NCCR for $k[\cone(Q\times\{1\})^\vee\cap (M\oplus \Z)]$, descending to give an NCCR of $R$.

\subsection{Structure and Notation}

We give a brief overview of toric Deligne-Mumford stacks and non-commutative crepant resolutions in \S\ref{Sec:Background}. Then, in \S\ref{Sec:Descent}, we introduce and prove the main technical result of this paper in Theorem \ref{Thm:FaceOfReflexiveTORICNCCR}. The theorem states that a toric NCCR for an affine toric Gorenstein algebra associated to a cone $\sigma$ descends to a toric NCCR for cones lattice equivalent to its faces $\tau\prec \sigma$. We apply this theorem in \S\ref{sec:Applications} to give altenative proofs for the existence of NCCRs for toric algebras associated to simplicial and almost simplicial Gorenstein cones.

 Let us here fix some notation used throughout this paper. Given a normal noetherian domain $R$, we say that a finitely generated $R$-module is \newterm{reflexive} if the canonical map $M\mapsto \Hom_R(\Hom_R(M,R),R)$ is an isomorphism. Such modules form a category, which we denote by $\operatorname{ref}R$.  For a reflexive $R$-algebra $\Lambda$, denote by $\operatorname{ref}(\Lambda)$ the category of $\Lambda$-modules which are reflexive as $R$-modules. Under the additional assumption that $R$ is commutative, an $R$-module $M$ is said to be \newterm{maximal Cohen-Macaulay} if $M_m$ is maximal Cohen-Macaulay for every maximal ideal $m$. We denote the category of such modules by $\operatorname{CM}R$ and will omit the word maximal in this paper. As detailed in \cite{IR08, IW14b} $M\in \operatorname{CM}R\Leftrightarrow \Ext^i(M,R)=0$ for all $i>0$.
 Given a toric variety $X$, we usually denote by $M$ its character lattice and by $N$ the cocharacter lattice, dual to $M$.The pairing $M\times N\rightarrow \Z$ naturally extends to a pairing of $M_\R:=M\otimes_\Z \R$ and $N_\R:=N\otimes_\Z\R$.  As such, the fan $\Sigma$ of $X$ lives inside $N_\R$. For a fan $\Sigma$ with ray $\rho\in \Sigma(1)$, we denote by $D_\rho$ the torus invariant Weil divisor associated to $\rho$. 
Finally, we fix an algebraically closed field $k$ of characteristic 0. 

\subsection{Acknowledgements}
The authors would like to thank Will Donovan for fruitful conversations. The first author is supported by Bejing National Science Foundation International Scientist Program IS25013 and the Beijing Postdoctoral Research Foundation.
The second author would like to acknowledge grants from Beijing Institute of Mathematical Sciences and Applications (BIMSA), the Beijing NSF BJNSF-IS24005, the China National Science Foundation (NSFC) NSFC-RFIS program W2432008. He would like to thank NSF AI Institute for Artificial Intelligence and Fundamental Interactions at Massachusetts Iinstitute of Technology (MIT) which is funded by the US NSF grant under Cooperative Agreement PHY-2019786. He would also like to thank China's National Program of Overseas High Level Talent for generous support.

\section{Toric geometry and non-commutative resolutions}\label{Sec:Background}
In this section, we review some notions and facts from toric geometry and define non-commutative resolutions.

\subsection{Cox stacks and cohomology computations}
Recall that a \newterm{Gorenstein} cone is a strongly convex rational polyhedral cone $\sigma\subset N_\R$ with primitive ray generators $u_\rho\in N$ for $\rho\in \sigma(1)$ such that there is an element $m_\sigma\in M$ with $\langle m_\sigma, u_\rho\rangle=1$ for all $\rho\in \sigma(1)$. 
Fix now a fan $\Sigma\subset N_\R$ where $\dim N=n$. To the fan $\Sigma$ we can not only associate a toric variety, but also a quotient stack $\mathcal{X}_\Sigma$, using a standard construction by Cox. Let $\nu$ be the set of primitive ray generators of $\Sigma(1)$, i.e. $\nu:=\{u_\rho\mid \rho\in \Sigma(1)\}$. Consider the vector space $\R^{|\nu|}$ with elementary $\Z$-basis $e_\rho$, indexed by the rays $\rho\in \Sigma(1)$. The \newterm{Cox fan} of $\Sigma$ is now defined to be \[
\operatorname{Cox}(\Sigma):=\{\cone(e_\rho\mid \rho\in \sigma)\vert\sigma\in \Sigma\}.
\]
This fan is a subfan of the standard fan for $\A^{|\nu|}$, and thus we can view the associated toric variety as an open subset of $\A^{|\nu|}$, denoted by $U_\Sigma:=X_{\operatorname{Cox}(\Sigma)}$. There is a well-known right-exact sequence

\begin{equation}\label{eqn:Coxseqpre}
 M\xrightarrow{f_\Sigma}\Z^{|\nu|}\xrightarrow{\pi}\operatorname{coker}(f_\Sigma)\rightarrow 0,
\end{equation} 
where $m\mapsto \sum_{\rho\in\Sigma(1)}\langle u_\rho,m\rangle e_\rho\nonumber$. Applying to this sequence the functor $\Hom(-,\mathbb{G}_m)$ yields
\[
1\rightarrow \Hom(\operatorname{coker}(f_\Sigma),\mathbb{G}_m)\xrightarrow{\hat{\pi}}\mathbb{G}_m^{|\nu|}\rightarrow\mathbb{G}_m^n.
\]

From this sequence we extract the group $S_\Sigma:=\Hom(\operatorname{coker}(f_\Sigma),\mathbb{G}_m)$, which acts on $U_\Sigma$ and allows us to define the \newterm{Cox stack} as the following quotient stack.
\begin{definition}
    \label{Def:Coxstack}
    Define the \newterm{Cox stack} associated to $\Sigma$ to be \[
    \mathcal{X}_\Sigma:=[U_\Sigma/S_\Sigma]
    \]
\end{definition}

Considering the Cox stack associated to a fan has an advantage over simply considering the associated toric variety in that the stack can be smooth even when the variety is not. A toric variety is smooth if and only if the corresponding fan is a smooth fan, whereas a Cox stack merely requires the fan to be simplicial (which happens for orbifold singularities in the toric variety).
\begin{theorem}[Theorem 4.12 in \cite{FK18}]
    \label{Thm:4.12FK18}
    If $\Sigma$ is simplicial, then $\mathcal{X}_\Sigma$ is a smooth Deligne-Mumford stack with coarse moduli space $X_\Sigma$. When $\Sigma$ is smooth (equivalently, the variety $X_\Sigma$ is smooth), $\mathcal{X}_\Sigma\cong X_\Sigma$.
\end{theorem}

Let us now recall a way to compute the cohomology of line bundles on such stacks. These computations are standard, and our recollection here is based on work by Borisov-Hua \cite{BH09} and Efimov \cite{Efimov14}; the results below can be found within these two references. From hereon, we fix $\Sigma\subset N_\R$ to be a complete simplicial fan, keeping $\dim N=n$. For any $I\subseteq \Sigma(1)$, we denote by $C_I$ the simplicial complex with vertex set $I$, consisting of those subsets $J\subseteq I$ forming the set of rays of some cone in $\Sigma$. For example, $|C_\emptyset|=\emptyset$ and $|C_{\Sigma(1)}|=S^{n-1}$. Given a vector $\mathbf{r}=\sum r_\rho e_\rho\in \Z^{|\Sigma(1)|}\subset \R^{|\Sigma(1)|}$, we denote by $\operatorname{Supp}(\mathbf{r})$ the set of rays $\rho\in \Sigma(1)$ such that $r_\rho<0$. Using a standard computation via toric \v{C}ech complexes, one obtains the following result on the cohomology of line bundles on $\mathcal{X}_\Sigma$.

\begin{proposition}
    \label{Prop:CohoCech}
    Let $\mathcal{L}$ be a line bundle on $\mathcal{X}_\Sigma$. Then \[
    H^i(\mathcal{L})=\bigoplus_{\substack{\mathbf{r}\in \Z^{|\Sigma(1)|},\\ \O(\sum_{\rho\in \Sigma(1)}r_\rho D_\rho)\cong \mathcal{L}}}\tilde{H}_{i-1}(|C_{\operatorname{Supp}(\mathbf{r})}|).
    \]
\end{proposition}

For the purpose of effectively determining the cohomology of line bundles, it is thus required to understand when the reduced homology of the simplicial complexes $C_I$ vanishes. \begin{definition}
    \label{Def:ForbiddenPoint}
    Given a subset $I\subseteq \Sigma(1)$ of non-vanishing homology $\tilde{H}_\bullet(|C_I|)\neq 0$, define the \newterm{forbidden point}
\[
q_I=-\sum_{\rho\not\in I}D_\rho\in \operatorname{Pic}_\R(\mathcal{X}_\Sigma).
\]
Then define the \newterm{forbidden cone} $F_I\subseteq \operatorname{Pic}_\R(\mathcal{X}_\Sigma)$ to be the set\[
F_I=q_I+\sum_{\rho\in I}\R_{\ge0}D_\rho-\sum_{\rho\not\in I}\R_{\ge0}D_\rho.
\]
\end{definition}
\noindent We note the following immediate Corollary of the Proposition \ref{Prop:CohoCech}.
\begin{corollary}
    \label{Cor:AcyclicForbidden}
    Let $\mathcal{L}$ be a line bundle on $\mathcal{X}_\Sigma$. The following are equivalent:
    \begin{enumerate}
        \item $H^{>0}(\mathcal{L})=0$;
        \item $\mathcal{L}$ does not belong to any forbidden cone $F_I$, $I\neq \emptyset$.
    \end{enumerate}
\end{corollary}

To classify which sets $I$ give complexes $C_I$ of non-vanishing reduced homology, we consider \newterm{primitive collections}.
\begin{definition}
    \label{Def:PrimColln}
    A non-empty subset $I\subseteq \Sigma(1)$ is called a \newterm{primitive collection} if it is not a set of boundary cones of any cone in $\Sigma$, but each proper subset $J\subseteq I$ is.
\end{definition}

\noindent Efimov \cite{Efimov14} proves the following result.
\begin{lemma}[Lemma 4.4 in \cite{Efimov14}]
    \label{Lem:RedHomUnionofPrim}
    Let $I\subseteq \Sigma(1)$ be a non-empty subset such that $\tilde{H}^{\bullet}(|C_I|)\neq 0$. Then $I$ is a union of primitive collections.
\end{lemma}

\subsection{Non-commutative crepant resolutions}
Consider now a Gorenstein cone $\sigma$ with associated toric variety $X_\sigma$. It is known that subdividing the cone $\sigma$ gives a crepant toric morphism if and only if the added rays lie in the Gorenstein plane $\{\langle m_\sigma,n\rangle=1\}$. Thus, any simplicial fan $\Sigma$ with support $|\Sigma|=\sigma$ such that for all $\rho\in\Sigma(1)$, the primitive ray generator $u_\rho$ lies in this Gorenstein plane, gives a crepant resolution via the smooth toric DM stack $\mathcal{X}_\Sigma$. These resolutions are of a geometric nature, and it remains a challenge to translate this into the language of \newterm{non-commutative crepant resolutions}. 

\noindent Let $R$ be a normal noetherian domain. \begin{definition}
    \label{Def:NCCR}
    A \newterm{non-commutative resolution} of $R$ is an $R$-algebra $\Lambda=\End_R(M)$ for $M\in \operatorname{ref}R$ such that $\gldim\Lambda<\infty$. Further assuming that $R$ is Gorenstein, we call $\Lambda$ \newterm{crepant} if additionally $\Lambda\in \operatorname{CM}R$. 
\end{definition}
Usually, we will abbreviate non-commutative crepant resolutions as NCCRs. 
Note that the notion of NCCRs does capture in an algebraic manner the geometric concept of crepant resolutions by passing from an affine variety to its associated coordinate ring. Details of this correspondence are detailed in \cite[\S 4]{VdB04}. The definition of NC(C)Rs can be understood for more general schemes by reducing to affine patches.

Whilst the notion of NCCRs encapsulates the nature of crepant resolutions, it is still unclear to what extent they exist. A natural first class of variety to consider is the class of affine Gorenstein toric varieties. Given a cone $\sigma$, we refer to the associated coordinate ring as \newterm{toric algebra}.

\begin{conjecture}
    \label{Conj:affinetoric}
    A Gorenstein toric algebra always has an NCCR.
\end{conjecture}

Most known results in this context give a special type of NCCR, called \newterm{toric NCCR}.
\begin{definition}
    \label{Def:ToricNCCR}
    An NCCR $\Lambda=\End_R(M)$, $M\in \operatorname{ref}R$ is called \newterm{toric} if $M$ is isomorphic to a sum of ideals.
\end{definition}
 
Constructions giving toric NCCRs often exploit links to tilting theory, expressing the NCCR as endormorphism algebra of a \newterm{tilting bundle}.

\begin{definition}
    \label{Def:tilt}
    Let $Y$ be a Noetherian scheme. A \newterm{partial tilting complex} $\mathcal{T}$ on $Y$ is a perfect complex such that $\Ext^i_Y(\mathcal{T},\mathcal{T})=0$ for $i\neq 0$. We say that $\mathcal{T}$ is \newterm{tilting} if it generates the derived category $\operatorname{D}_{Qcoh}(Y)$ in the sense that it has trivial right orthgonal, i.e. $\Rhom_Y(\mathcal{T},\mathcal{F})=0$ implies $\mathcal{F}=0$. If the perfect complex is in fact a bundle, we refer to it as \newterm{(partial) tilting bundle}.
\end{definition}
\noindent This definition naturally generalises to algebraic stacks. 

Consider a noetherian scheme $Y$ with tilting complex $\mathcal{T}$ and let $\Lambda=\End_{Y}(\mathcal{T})$. The functor $\Rhom_{Y}(\mathcal{T},-)$ defines an equivalence between $\operatorname{D}_{Qch}(Y)$ and $\operatorname{D}(\Lambda^\circ)$ \cite[Theorem 1.2]{KK20}. Assuming further that $Y$ is regular and $\Lambda$ is right noetherian, then $\Lambda$ has finite global dimension and we obtain an equivalence of categories $\dbcoh{Y}\cong \dbmod\Lambda$. This generalises to the context of smooth DM stacks $\mathcal{Y}$. It is this derived equivalence between the endomorphism algebra of tilting objects and the derived category of coherent sheaves of a bonafide crepant resolutions that motivates the notion of NCCR. 

Using tilting theory, NCCRs have been obtained for a variety of cases, see for instance \cite{SVdBtoricI,SvdB17,MS25,svdbt2}. One key result in this context from previous work of Malter-Sheshmani we would like to point out is the following.
\begin{theorem}[Theorem 4.11 in\cite{MS25}]
    \label{Thm:NovaDMstack}
     Let $\Sigma$ be a complete simplicial fan such that $\mathcal{X}_\Sigma$ has a tilting complex $\mathcal{T}$. Let $p:\tot V\rightarrow X_\Sigma$ be a toric vector bundle on $X_\Sigma$ with fan $\mathcal{V}$. Let $f_\Sigma:\mathcal{X}_\Sigma\rightarrow X_\Sigma$ be the good moduli space. If $H^i(\mathcal{X}_\Sigma, \mathcal{T}^\vee\otimes\mathcal{T}\otimes \operatorname{Sym}^\bullet(f_\Sigma^\ast V^\vee))=0$ for all $i\neq 0$, then there is a tilting complex on $\mathcal{X}_{\mathcal{V}}$. If $V=\omega_{X_{\Sigma}}$, the ring $R=k[|\mathcal{V}|^\vee\cap M]$ then has an NCCR.
\end{theorem}

Other, more algebraic, tools to construct NCCRs include dimer models and conic/divisorial modules. Using dimer models, Broomhead \cite{Broomhead} proved the existence of NCCRs for three-dimensional Gorenstein toric algebras, whereas conic modules have been used to construct NCCRs for all simplicial toric algebras by Faber, Muller and Smith \cite{FMS19}. Recently Tomonaga \cite{Tomo25} used divisorial modules and the combinatorics of upper sets to construct NCCRs for Gorenstein toric algebras coming from cones $\sigma$ with $|\sigma(1)|=\dim \sigma+1$.

\section{Descending NCCRs}\label{Sec:Descent}

The main technical result of this paper concerns the descent of toric NCCRs in the following sense. 
Up to a change of basis, all Gorenstein cones can be expressed as $\sigma'=\cone(Q\times\{1\})$ for some lattice polytope $P$. Given an NCCR for the toric algebra of such a cone $\sigma'$ and a face $P$ of $Q$, geometrically we expect an NCCR for the toric algebra associated to $\sigma=\cone(P\times\{1\})$ to exist. Note here that the two cones lie in distinct spaces $N'_\R$ and $N_\R$, having different dimensions. The geometric expectation comes from the fact that if we obtain the NCCR via a tilting object associated to a simplicial subdivision $\Sigma'$ of $\sigma'$, the regular triangulation of $Q$ giving that subdivision descends to a regular subdivision of $P$. The toric stack coming from the restriction of $\Sigma'$ to $\sigma$ is thus a smooth toric DM stack that provides a crepant resolution of $X_\sigma$. Unfortunately, tilting is not a local property and we cannot simply presume the existence of a tilting object. However, for the case of a toric NCCR, we are able to prove the desired descent in the following Theorem.
\begin{theorem}
    \label{Thm:FaceOfReflexiveTORICNCCR}
    Let $Q\subset \R^k$ be a lattice polytope and let $\sigma$ be the cone $\sigma=\cone(Q\times\{1\})\subset \R^{k+1}$ together with its toric algebra $R_\sigma=k[\sigma^\vee\cap M_\sigma]$, where $M_\sigma$ is the character lattice of the toric variety defined by $\sigma$. Suppose $Q$ is lattice equivalent to a face $F$ of a lattice polytope $P\subset \R^n$. Consider the cone $\sigma'=\cone(P\times\{1\})\subset\R^{n+1}$. Denoting the corresponding character lattice by $M'$, let $R'=k[(\sigma')^\vee\cap M']$ be the toric algebra associated to $\sigma'$. Suppose $R'$ has a toric NCCR $\Lambda'$. Then $R$ has a toric NCCR.
\end{theorem}

\begin{proof}

    Observe that the toric algebra $R_F=k[\sigma_F^\vee\cap M']$ associated to the cone $\sigma_F=\cone(F\times\{1\})$ is obtained by localising the algebra $R'$ on some multiplicative set $S$, i.e. $R_F=S^{-1}R'$. Specifically, this is the content of \cite[Proposition 1.3.16]{CLS}: $R_F$ is the localisation of $R'$ at $\chi^m\in R'$, where $m\in (\sigma')^\vee\cap M'$ is the lattice element such that $\sigma_F=\sigma'\cap H_m$ ($H_m$ is the hyperplane $\{u\in (N')_\R\mid\langle m,u\rangle=0\}\subseteq (N')_\R$). The proof of Theorem \ref{Thm:FaceOfReflexiveTORICNCCR} now splits into two parts. Firstly, we show that $R_F$ has an NCCR and then we relate $R_F$ to $R$ and show that from there we can deduce the existence of an NCCR for $R$.

    Write the NCCR $\Lambda'$ of $R'$ as $\Lambda'=\End_{R'}(M')$ with $M'\in \operatorname{ref}R$. Then $\Lambda'\in \operatorname{CM}R$ and $\gldim\Lambda'<\infty$. For the first part of the proof, we shall show that $S^{-1}\Lambda'=\End_{R_F}(S^{-1}M)$ where $S^{-1}M\in \operatorname{ref}R_F$, $S^{-1}\Lambda'\in \operatorname{CM}R_F$ and $\gldim(S^{-1}\Lambda')<\infty$.
    We begin by noting that the functor $S^{-1}$ is an exact functor \cite[Proposition 3.3]{AtiyahMcD}. 
    
    Since $\Lambda'$ is Noetherian, any $S^{-1}\Lambda'$-module arises as module of the form $S^{-1}M$ for some $\Lambda'$-module $M$. Consider a projective resolution \[0\rightarrow P^m\rightarrow\dots\rightarrow P^1\rightarrow M\rightarrow 0\] of $M$ as $\Lambda'$-module. Since $d:=\gldim(\Lambda')<\infty$, we have $m<d$.  
    
    Apply the functor $S^{-1}$ to it, which is exact and thus gives an exact sequence \[
    0\rightarrow S^{-1}P^m\rightarrow S^{-1}P^{m-1}\rightarrow\dots\rightarrow S^{-1}P^1\rightarrow S^{-1}M\rightarrow 0.
    \]
    Each of the modules $S^{-1}P^i$ is projective as $S^{-1}\Lambda'$-module \cite[Lemma 10.109.13]{stacks}.

    Now, let us show that $S^{-1}\Lambda'$ is a trivial Azumaya algebra, i.e. of the form $\End_{S^{-1}R'}(M'')$ for some $M''\in \operatorname{ref}S^{-1}\Lambda'$. The NCCR $\Lambda'$ itself is of the form $\End_{R'}(M')$ with $M'\in\operatorname{ref} R'$. Note that as the module $M'$ is finitely generated and $R'$ is Noetherian, $M'$ is finitely presented. For a multiplicative set $S$ over the Noetherian ring $R'$, if $M'$ is finitely presented, localisation commutes with the functor $\Hom(M,-)$. 
   
    Hence, \[
    S^{-1}\End_{R'}(M')=\End_{S^{-1}R'}(S^{-1}M').
    \]

    It remains to show that $M'':=S^{-1}M'\in \operatorname{ref}S^{-1}R'$, noting that the rank of a module is preserved under localisation. Since $R'$ is Noetherian, the dual of a finitely generated module is finitely presented, and so $M'\in\operatorname{ref}R'$ implies that $\Hom_{R'}(M',R')$ is finitely presented. Being finitely presented localises (as the localisation functor is exact) and so we have
    \begin{align*}
        S^{-1}M' &=S^{-1}\Hom_{R'}(\Hom_{R'}(M',R'),R')\\
        &=\Hom_{S^{-1}R'}(S^{-1}\Hom_{R'}(M',R'),S^{-1}R')\\
        & = \Hom_{S^{-1}R'}(\Hom_{S^{-1}R'}(S^{-1}M',S^{-1}R'),S^{-1}R').
    \end{align*}
    Hence, $M''\in\operatorname{ref}S^{-1}R'$.
    Finally, we aim to show that $S^{-1}\Lambda'\in\operatorname{CM}S^{-1}\Lambda'$. The algebra $S^{-1}R'$ is commutative Noetherian, as $R'$ is. Since the functor $S^{-1}$ is exact, for any complex $C^\bullet$, we have $S^{-1}H_i(C^\bullet)=H_i(S^{-1}C^\bullet)$. 
    
    Applying this, we obtain \[
    \Ext_{R'}^i(\Lambda',R')=\Ext_{S^{-1}R'}^i(S^{-1}\Lambda',S^{-1}R').
    \]
    Since $\Lambda'\in\operatorname{CM}R'$, the left-hand side vanishes for $i>0$, and thus so does the right-hand side. In turn, this implies that $S^{-1}\Lambda'\in\operatorname{CM}S^{-1}R'$, as required. So $S^{-1}\Lambda'$ is a toric NCCR for $R_F=S^{-1}R'$ when $\Lambda'$ is for $R'$. 
    
   \noindent Next, let us relate $R_F$ and $R$, following closely the proof of \cite[Theorem 1.3.12]{CLS}. Denote by $N_1$ the smallest saturated sublattice of $N'$ (the dual lattice to $M'$) containing the generators of the cone $\sigma_F$. The quotient $N'/N_1$ is torsion-free and so $N'=N_1\oplus N_2$, giving $M'=M_1\oplus M_2$ on the level of dual lattices. Let $\sigma_{F,N_1}$ be the cone $\sigma_F$ considered within the lattice $N_1$. Then for the relevant semigroups we have $(\sigma_F)^\vee\cap M'=(\sigma_{F,N_1})^\vee\cap M_1\oplus M_2$, giving \[R_F\simeq k[(\sigma_{F,N_1})^\vee\cap M_1]\otimes_k k[M_2]. \]
    
    But now by the lattice equivalence of the face $F$ of $P$ to $Q$, we see that $R\simeq k[(\sigma_{F,N_1})^\vee\cap M_1]$, and so $R_F\simeq R\otimes_k k[M_2]$. Note that the algebra $k[M_2]$ is the coordinate ring of a torus of dimension $n-k$. Since both $R$ and $k[M_2]$ are integral domains, we have, by a result of Ischebeck's \cite[Theorem 1.7]{Ischebeck}, the following exact sequence\[
    0\rightarrow \operatorname{Pic}(\spec R)\oplus \operatorname{Pic}(\spec k[M_2])\rightarrow \operatorname{Pic}(\spec(R\otimes k[M_2]))\rightarrow \operatorname{Pic}(k(\spec R)\otimes k(\spec k[M_2])).
    \]
    Since the torus has trivial class-group, this reduces to the exact sequence
    \[
    0\rightarrow \operatorname{Pic}(X_\sigma)\rightarrow \operatorname{Pic}(X_\sigma\times (k^\times)^{n-k})\rightarrow \operatorname{Pic}(k(X_\sigma)\otimes k((k^\times)^{n-k})).
    \]
    In particular, the tensor product $k(X_\sigma)\otimes k((k^\times)^{n-k})$ can be written as fraction ring $T^{-1}A$ of the ring $A=k(X_\sigma)[z_1,z_2,\dots,z_{n-k}]$ over the field $k(X_\sigma)$ and is thus a UFD, with trivial Picard group.

    Thus, $\operatorname{Pic}(X_\sigma\times (k^\times)^{n-k})\cong \operatorname{Pic}(X_\sigma)\oplus \operatorname{Pic}((k^\times)^{n-k})\cong\operatorname{Pic}(X_\sigma)$. In particular, all rank 1 modules of $R\otimes k[M_2]$ decompose as tensor products of rank 1 modules over $R$ and $k[M_2]$, respectively. Since $\operatorname{Pic}(k[M_2])=0$, we have that the unique rank 1 module (up to isomorphism) is simply $k[M_2]$ itself. Thus, we have $M''$ is isomorphic to a direct sum of modules of the form $M_R\otimes k[M_2]$, where $M_R$ is a rank 1 module over $R$. Furthermore, $M_R\in \operatorname{ref}R$ due to the flatness of $-\otimes k[M_2]$.
   
    The flatness and finite presentation of $M_R$ also gives\[
    \End_{R\otimes k[M_2]}(M_R\otimes k[M_2])\cong \End_R(M_R)\otimes k[M_2].
    \]
    A K\"unneth-type result now yields
    \[
    \Rhom_R(M_R,R)\otimes_k \Rhom_{k[M_2]}(k[M_2],k[M_2])\simeq \Rhom_{R\otimes k[M_2]}(p_1^\ast M_R\otimes p_2^\ast k[M_2], p_1^\ast R\otimes p_2^\ast k[M_2]),
    \]
    where $p_1, p_2$ are the natural projections from $R\otimes k[M_2]$ to  $R$ and $k[M_2]$.
    Taking $H^i$ on both sides gives 
    \[\Ext^i_{R\otimes k[M_2]}(\End(M_R)\otimes k[M_2],R\otimes k[M_2])\cong \bigoplus_{p+q=i}\Ext^p_R(\End_R((M_R),R))\otimes \Ext^q(k[M_2],k[M_2]).\]
    Since the left-hand side vanishes for $i>0$, the equality to the right-hand side implies $\Ext^i(\End_R(M_R),R)=0$ for $i>0$ and thus $\End_R(M_R)\in \operatorname{CM}R$. 

    Finally, we note that  $\gldim \End_R(M_R)<\infty$ as $\End_{R\otimes k[M_2]}(M_R\otimes k[M_2])$ is finite. Hence, $\End_R(M_R)$ is a toric NCCR of $R$.
    
\end{proof}

The proof above relied on descending rank 1 reflexive modules on $R\otimes k[M_2]$ to reflexive modules on $R$, which does not happen in the same way for higher rank. Hence, the proof from that point onwards needs adjustment when considering NCCRs of a non-toric nature. We do still believe the statement to hold and thus reiterate the following conjecture.
\begin{conjecture}
    \label{Conj:FaceRkBigger1}
The statement of Theorem \ref{Thm:FaceOfReflexiveTORICNCCR} holds also if $\Lambda'$ is a non-toric NCCR. 
\end{conjecture}

\section{Application to (almost) simplicial toric algebras}\label{sec:Applications}

Heuristically, as detailed in \cite{MS25}, Theorem \ref{Thm:FaceOfReflexiveTORICNCCR} and Conjecture \ref{Conj:FaceRkBigger1} allows us to reduce the Conjecture \ref{Conj:affinetoric} to the case of reflexive Gorenstein cones. Indeed, to do so we note the following well-known fact of toric geometry.
\begin{proposition}[Proposition 2.2 in \cite{HM06}]
    \label{Prop:FaceOfReflexive}
    Let $P$ be a lattice polytope. Then $P$ is lattice equivalent to a face of some reflexive polytope $Q$.
\end{proposition}

In \cite{MS25}, the authors prove the following result (the word toric was omitted but the proof shows the toric nature of the NCCR as it derives from a full strong exceptional collection).
\begin{theorem}[Theorem 5.1 in \cite{MS25}]
    \label{Thm:FanoAlmostSimplicial}
    Let $P\in N_{\R}\cong \R^n$ be a simplicial, reflexive polytope with $\le n+2$ vertices. Consider the cone $\sigma=\cone(P\times\{1\})$. Then $R=k[\sigma^\vee\cap M]$ has a toric NCCR.
\end{theorem}

An immediate consequence of Theorem \ref{Thm:FaceOfReflexiveTORICNCCR}, Theorem \ref{Thm:FanoAlmostSimplicial} and Proposition \ref{Prop:FaceOfReflexive} is the following short re-proof of the fact that all simplicial Gorenstein cones admit toric NCCRs. The existence of such NCCRs for all simplicial cones has been proven in \cite{FMS19}.

\begin{corollary}
    \label{Cor:SimplicialHasNCCR}
    Let $\sigma\subset N_{\R}$ be a simplicial Gorenstein cone. Then $R=k[\sigma^\vee\cap M]$ has an NCCR.
\end{corollary}

\begin{proof}
    Write $\sigma$ as $\cone(P\times\{1\})$ for a simplicial polytope $P$. By Proposition \ref{Prop:FaceOfReflexive} there exists a reflexive polytope $Q$ with a face that is lattice equivalent to $P$. Thus, it is sufficient to prove that for $\sigma'=\cone(Q\times\{1\})$, there is an NCCR of $k[(\sigma')^\vee\cap M']$, where $M'$ is the corresponding character lattice of the affine toric variety $X_{\sigma'}$. The proof of Proposition \ref{Prop:FaceOfReflexive} constructs $Q$ from $P$ by repeated application of the wedge construction over a codimension 1 facet. Since $P$ is simplicial, the number of vertices at each step increases by 1 and so $Q$ is simplicial (see cf. \cite[Chapter 15]{HandbookDiscrete}). Then $k[(\sigma')^\vee\cap M']$ admits a toric NCCR by Theorem \ref{Thm:FanoAlmostSimplicial}, and Theorem \ref{Thm:FaceOfReflexiveTORICNCCR} guarantees the existence of an NCCR of $R$, as desired.
\end{proof}

Corollary \ref{Cor:SimplicialHasNCCR} also follows as a consequence of the work by Ballard et al. \cite{BBB+} and can be considered well-established in the literature. The next case to be considered is that of toric varieties associated to cones $\sigma$ with Picard rank 1, i.e. where the cone has $\dim\sigma+1$ extremal rays. We refer to such cones as \newterm{almost simplicial}. 
\begin{definition}
    \label{Def:AlmSim}
    We say that a cone $\sigma$ is \newterm{almost simplicial} if $|\sigma(1)|=\dim \sigma+1$. 
\end{definition}

\noindent Using Theorem \ref{Thm:FaceOfReflexiveTORICNCCR}, we will prove the following result.
\begin{theorem}
    \label{Thm:AlmSim}
    Let $\sigma\subset (N\oplus \Z)_\R$ be an almost simplicial Gorenstein cone, i.e. $\sigma=\cone(P\times\{1\})$ with $P\subset N_\R$ a lattice polytope with $\dim P+2$ vertices. Then the toric algebra $R=k[\sigma^\vee\cap (M\oplus \Z)]$ associated to the cone $\sigma$ has a toric NCCR.
\end{theorem}

This result has also recently been proven by Tomonaga \cite{Tomo25}. Tomonaga's proof uses a partial order on upper sets, and is algebraic in nature. The proof we will give below resembles the proof we give of Corollary \ref{Cor:SimplicialHasNCCR} and is thus more geometrically intuitive. 
Let us first outline the structure of the proof to Theorem \ref{Thm:AlmSim}. Given a polytope $P$ with $\dim P+2$ vertices, we first note that if $P$ is a simplicial polytope (i.e. all its faces are simplices) then an NCCR exists by Theorem \ref{Thm:FanoAlmostSimplicial}. Otherwise, we construct a primitive polytope $Q$ such that $0\in \operatorname{Int}(Q)$ and $P$ is a face of $Q$. It then suffices to provide the existence of a toric NCCR for the toric algebra associated to $\cone(Q\times\{1\})$. The given construction of $Q$ lends itself to an explicit simplicial subdivision $\Sigma$ of its face fan for which we find the primitive collections of rays, and thus the forbidden cones. By modifying the arguments in \cite[\S 5]{BH09}, we obtain a tilting bundle on $\mathcal{X}_\Sigma$ fulfilling the required cohomology vanishing condition of Theorem \ref{Thm:NovaDMstack}, and thus we have a toric NCCR as desired.

Fix now a lattice polytope $P\subset N_\R$ with $\dim P+2$ vertices. We begin by constructing an appopriate polytope $Q\subset (N\oplus \Z)_\R$ containing a face lattice equivalent to $P\subset N_\R$. Let $n=\dim N$ and consider a standard basis $e_1,\dots, e_n$ for $N_\R$ that we extend via $e_{n+1}$ to a basis of $(N\oplus \Z)_\R$. Note that there exists a regular triangulation of $P$ obtained by introducing a hyperplane. This hyperplane $H$ will separate two vertices, which we denote by $w_1$ and $w_2$, with the other vertices spanning the hyperplane denoted by $w_3,\dots,w_{n+2}$. Now place the polytope $P$ at height $-1$ in $(N\oplus \Z)_\R$ such that the vertex $w_1$ lies at $-e_{n+1}$.\\

Since $P$ is convex and $w_1$ a vertex, either $e_1-e_{n+1}$ or $-e_1-e_{n+1}$ is not in the polytope, and we may without loss of generality assume $-e_1-e_{n+1}$ is not in the polytope. For each vertex $w_i$ of $P\subset N_\R$, we denote the corresponding vertex of the placed polytope lying at height $-1$ by $v_i$. 
There exists a minimal positive integer $k_0\in \Z_{>0}$ such that the origin and $v_1$ lie on the same side (and not inside) of the hyperplane spanned by  $\{v_i\mid 3\le i\le n+2\}\cup \{-e_1+k_0 e_{n+1}\}$. Consider now the polytope $Q$ given by the convex hull of $v_1,\dots, v_{n+2}$ and $v_{n+3}=-e_1+k_0e_{n+1}$. The faces of $Q$ are given by $\{F\times \{-1\}\mid F\preceq P\}\cup\{\conv(F, v_{n+3})\mid F\prec P\}$. The polytope $Q$ is primitive, since all its vertices are primitive lattice elements.\\

Let $\Sigma'$ be the face fan of $Q$. It is non-simplicial, but introducing the cone $\cone(v_3,\dots,v_{n+2})$ simplicially subdivides it (since introducing $\conv(w_3,\dots,w_{n+2})$ gives a regular triangulation of $P$). Denote this simplicial fan by $\Sigma$. Let $\rho_i$ be the ray with primitive generator $v_i$. Then we can list the maximal cones (of which there are $2n+2$) of $\Sigma$ via their set of extremal rays: \begin{itemize}
    \item $\{\rho_1,\rho_3,\dots,\rho_{n+2}\}$;
    \item $\{\rho_2,\rho_3,\dots,\rho_{n+2}\}$;
    \item $\{\rho_1,\rho_3,\dots,\rho_{n+3}\}\setminus\{\rho_i\}$ for $3\le i\le n+2$;
    \item $\{\rho_2,\dots,\rho_{n+3}\}\setminus \{\rho_i\}$ for $3\le i\le n+2$.
\end{itemize}
Note here that $\{\rho_3,\dots,\rho_{n+3}\}$ does not give a cone in $\Sigma$ as the hyperplane $H$ separates $w_1$ and $w_2$ and so the ray $\rho_2$ star subdivides the cone these  $n+1$ rays would span. The figure \ref{fig:Q} illustrates on an example the polytope $Q$ and the fan $\Sigma$.\footnote{And the authors acknowledge the help of AI based language models in generating this figure.}

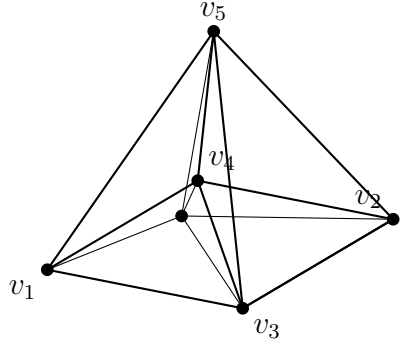
\begin{figure}[ht!]
\begin{tikzpicture}[tdplot_main_coords, scale=2]
\coordinate (v4) at (0,0,0);
\coordinate (v1) at (2,0,0);
\coordinate (v3) at (2,1.5,0);
\coordinate (v2) at (0,1.5,0);
\coordinate (v5) at (1,0.7,1.5);
\coordinate (p) at (1.6,0.8,0.4);
\draw[thick]
  (v1) -- (v3) -- (v2) -- (v4) -- cycle;
\draw[thick] (v3) -- (v4)
             (v2) -- (v3);
\draw[thick]
  (v5) -- (v1)
  (v5) -- (v2)
  (v5) -- (v3)
  (v5) -- (v4);
\draw
  (p) -- (v1)
  (p) -- (v2)
  (p) -- (v3)
  (p) -- (v4)
  (p) -- (v5);
\fill (v1) circle (1.2pt) node[below left] {$v_1$};
\fill (v3) circle (1.2pt) node[below right] {$v_3$};
\fill (v4) circle (1.2pt) node[above right] {$v_4$};
\fill (v2) circle (1.2pt) node[above left] {$v_2$};
\fill (v5) circle (1.2pt) node[above] {$v_5$};
\fill (p) circle (1.2pt);
\end{tikzpicture}
\caption{Example of a polytope $Q$ with fan $\Sigma$ for $P=\conv(v_1,v_2,v_3,v_4)$.}
\label{fig:Q}
\end{figure}

We now aim to describe the forbidden cones for $\Sigma$, so that we may classify the line bundles with vanishing higher cohomology. Using Lemma \ref{Lem:RedHomUnionofPrim}, we thus need to classify the primitive collections $I\subseteq \Sigma(1)$.

\begin{lemma}
    \label{Lem:AlmSimPrim}
    Let $\Sigma$ be as above. Then there exist precisely two primitive collections $I\subset \Sigma(1)$: $I_+:=\{\rho_1,\rho_2\}$ and its complement $I_-:=\{\rho_1,\rho_2\}^c$. Thus, there are exactly three proper subsets of $\Sigma(1)$ giving forbidden cones: $\emptyset, I_-$ and $ I_+$.
\end{lemma}
\begin{proof}
    Using the above list of maximal cones, it is immediately clear that these two sets are primitive collections, so it suffices to show no others exist. Since $\{\rho_1,\rho_2\}$ is primitive, no $I\subset \Sigma(1)$ containing both $\rho_1,\rho_2$ can be primitive. Similarly, if neither is in the set $I$, then it is not primitive unless it is the full set $\{\rho_1,\rho_2\}^c$. Hence, we are left with showing that no primitive collection can contain exactly one of $\rho_1, \rho_2$. Suppose $\rho_1\in I\subset\Sigma(1)$ and $I$ is a primitive collection. Then $\rho_2\not\in I$ and so $I=\{\rho_1\}\cup J$, $J\subsetneq\{\rho_1,\rho_2\}^c$ (the inequality follows as $\{\rho_1,\rho_2\}^c$ is already primitive). But for any $J\subsetneq \{\rho_1,\rho_2\}^c$, the set $I=\{\rho_1\}\cup J$ determines a cone in $\Sigma$ and so the set $I$ is not a primitive collection. Similar argumentation shows that no primitive collection $I$ contains $\rho_2$ but not $\rho_1$ and therefore the primitive collections are exactly $I_+$ and its complement $I_-$.
\end{proof}

Next, we will construct a full, strong exceptional collection $\mathcal{S}$ of line bundles on $\mathcal{X}_\Sigma$. This construction is only a minor modification of the construction in \cite[\S 5]{BH09}.

\begin{lemma}
    \label{Lem:collectionsrandalpha}
    There exists a collection of positive numbers $r_i$ such that $\sum_{i=1}^{n+3}r_i v_i=0$ and $\sum_{i=1}^{n+3}r_i=1$. There also exists a, unique up to scaling, collection of rational numbers $\alpha_i$ such that $\sum_{i=1}^{n+3}\alpha_i v_i=0$ and $\sum_{i=1}^{n+3}\alpha_i=0$. Furthermore, $\alpha_{n+3}=0$. 
\end{lemma}
\begin{proof}
Begin by noting that $0$ is in the interior of $Q$, implying the existence of the collection $r_i$ directly. 
To find the collection $\alpha_i$, note that the fan $\Sigma$ is complete and its primitive vertices $v_i$ generate $N\otimes \Q$, and so the space of linear relations has dimension two. The collection of $r_i$ above has $\sum_{i=1}^{n+3}r_i>0$, so the condition $\sum_{i=1}^{n+3}\alpha_i=0$ cuts out a dimension one subspace of the linear relations. We claim now that we can pick this collection such that $\alpha_1,\alpha_2>0$, $\alpha_3,\dots,\alpha_{n+2}<0$ and $\alpha_{n+3}=0$. Note that the line connecting $w_1$ and $w_2$ intersects the hyperplane $H$ in a unique point $z\in \operatorname{Relint}(H)$.

Thus, there are positive coefficients $\alpha_1, \alpha_2>0$ with $\alpha_1w_1+\alpha_2 w_2=z$ and $\alpha_1+\alpha_2=1$ as well as coefficients $\beta_i>0$ such that $\sum_{i=3}^{n+2}\beta_i w_i=z$ and $\sum_{i=3}^{n+2}\beta_i=1$. The image $\tilde{z}$ of the point $z$ in the face of $Q$ obtained by placing $P$ has the same property, i.e. $\tilde{z}=\alpha_1 v_1+\alpha_2 v_2=\sum_{i=3}^{n+2}\beta_i v_i$. Write $\alpha_i=-\beta_i$ for $3\le i\le n+2$ to obtain $\sum_{i=1}^{n+2}\alpha_i v_i=0$ and $\sum_{i=1}^{n+2}\alpha_i=0$. Adding $\alpha_{n+3}=0$ we obtain a collection of $\alpha_i$ with $\sum_{i=1}^{n+3}\alpha_i v_i=0$ and $\sum_{i=1}^{n+3}\alpha_i=0$, and so by uniqueness up to scaling our claim holds.
\end{proof}

Define on $\operatorname{Pic}_\R(\mathcal{X}_{\Sigma})$ the linear functions $f$ and $\alpha$ by $f(D_{\rho_i})=r_i$ and $\alpha(D_{\rho_i})=\alpha_i$. Consider the parallelogram $\Delta\subset \operatorname{Pic}_\R(\mathcal{X}_\Sigma)$ given by the inequalities $|f(x)|\le \frac{1}{2}$ and $|\alpha(x)|\le \frac{1}{2}(\alpha_1+\alpha_2)$. Let $p$ be a generic point $p\in \operatorname{Pic}_\R(\mathcal{X}_\Sigma)$ such that no points from $\operatorname{Pic}_\Q(\mathcal{X}_\Sigma)$ lie on the sides of $p+\Delta$.

\begin{proposition}
    \label{Prop:ExcColln}
    Let $\mathcal{S}$ consist of all those line bundles such that their image in $\operatorname{Pic}_\R(\mathcal{X}_\Sigma)$ lies inside this parallelogram $p+\Delta$. Then $\mathcal{S}$ forms a full strong exceptional collection on $\mathcal{X}_\Sigma$.
\end{proposition}

At this stage, we have finished the modifications to the construction by Borisov-Hua \cite{BH09} and the proof of \cite[Theorem 5.11]{BH09} applies verbatim to the situation we described. Since we do require specific properties of this construction to prove Theorem \ref{Thm:AlmSim}, we include the explicit proof here so that this paper is self-contained.
\begin{proof}
     Observe that the three forbidden points $q_{\emptyset},$ $q_{I_-}$ and $q_{I_+}$ lie on the boundary of the parallelogram $2\Delta$: $f(q_{\emptyset})=-1,$ $\alpha(q_{\emptyset})=0$, $f(q_{I_\pm})\in (-1,0)$ and $\alpha(q_{I_-})=1$ while $\alpha(q_{I_+})=-1$. The side of $2\Delta$ on which the forbidden point $q_I$ lies for $I\in \{\emptyset, I_+, I_-\}$ is in fact a supporting hyperplane of the corresponding forbidden cone. For instance, if $x\in F_{I_-}$ then \[
\alpha(x)=-\sum_{\rho\in I_-}\alpha_i+\sum_{\rho\in I_+}t_i\alpha_i-\sum_{j\in I_-}t_j\alpha_j\ge -\sum_{i\in I_-}\alpha_i,
\]
and similarly for the two other cones. Thus, no point in the interior of $2\Delta$ lies in either of the three forbidden cones.

To show that the collection $\mathcal{S}$ is strong exceptional is now straightforward. The line bundles correspond to points in the interior of $p+\Delta$ and so the difference between any two points lies in the interior of $2\Delta$. Consequently, for any $\mathcal{L}_1, \mathcal{L}_2\in S$, and so $\Ext^i(\mathcal{L}_1,\mathcal{L}_2)=H^i(\mathcal{X}_\Sigma,\mathcal{L}_1^\vee\otimes \mathcal{L}_2)=0$ for $i>0$ (as its image lies outside all forbidden cones).

Showing that the collection $\mathcal{S}$ is full takes a little more work. 
Divide the sides of $\Delta$ determined by constant $\alpha$ into two line segments each. Consider the four points \[
\pi_{\pm}=\pm\frac{1}{2}\sum_{\rho\in I_+}D_\rho,\quad \mu_{\pm}=\pm \frac{1}{2}\sum_{\rho\in I_-}D_\rho.
\]
Then $\alpha(\pi_{\pm})=\pm (\alpha_1+\alpha_2)$ and $\alpha(\mu_\pm)=\mp (\alpha_1+\alpha_2)$ while \[-\frac{1}{2}=-\frac{1}{2}\sum_{\rho_i\in \Sigma(1)}r_i<f(\pi_{\pm}), f(\mu_{\pm})<\frac{1}{2}=\frac{1}{2}\sum_{\rho_i\in\Sigma(1)}r_i. \]
Thus, $\pi_+, \mu_-$ lie on the side of $\Delta$ determined by $\alpha(x)=\frac{1}{2}(\alpha_1+\alpha_2)$ while $\pi_-,\mu_+$ lie on the side determined by $\alpha(x)=-\frac{1}{2}(\alpha_1+\alpha_2)$. Furthermore, \[
f(\pi_+)=\frac{1}{2}\sum_{\rho_i\in I_+}r_i=\frac{1}{2}-\frac{1}{2}\sum_{\rho_i\in I_-}r_i=\frac{1}{2}+\frac{1}{2}f(\mu_-),
\]
and so $f(\pi_+)-f(\mu_-)=\frac{1}{2}$. Thus, we can form two line segments $\theta_{+}^+,$ $\theta^+_-$ subdividing the side of $\Delta$ determined by $\alpha(x)=\frac{1}{2}(\alpha_1+\alpha_2)$ such that $\pi_+$ is the midpoint of $\theta^+_+$ and $\mu_-$ of $\theta_-^+$. Similarly we obtain two line segments $\theta^-_+$ and $\theta^-_-$ subdividing the side of $\Delta$ determined by $\alpha(x)=-\frac{1}{2}(\alpha_1+\alpha_2)$ with midpoints $\pi_-$ and $\mu_+$ respectively. We observe that the line segments $\theta_-^{\pm}$ have the same length, and the same holds for $\theta_+^{\pm}$. Figure \ref{fig:Parallelo} shows the parallelograms $\Delta$ and $2\Delta$ with the points sitting on their sides as well as the line segments $\theta_{\pm}^\pm$.

\tikzset{every picture/.style={line width=0.75pt}} 
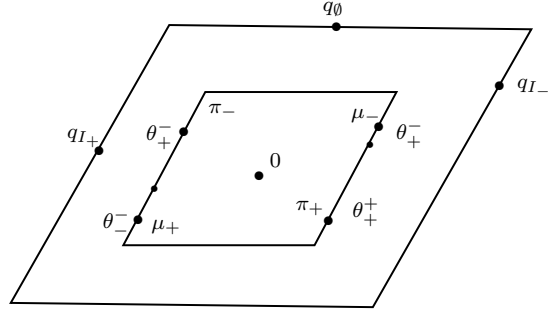
\begin{figure}[ht!]
\begin{tikzpicture}[x=0.75pt,y=0.75pt,yscale=-1,xscale=1]

\draw   (128.05,283.36) -- (310.67,285.5) -- (230.75,425.75) -- (48.13,423.61) -- cycle ;

\draw   (146.23,317.17) -- (242.67,317.17) -- (201.34,394.5) -- (104.9,394.5) -- cycle ;

\draw (13.67,363.33) node   [align=center] {\begin{minipage}[lt]{158pt}\setlength\topsep{30pt}
$ $
\end{minipage}};

\draw (204,270.07) node [anchor=north west][inner sep=0.75pt]  [font=\tiny]  {$q_{\emptyset}$};

\draw (75,334.4) node [anchor=north west][inner sep=0.75pt]  [font=\tiny]  {$q_{I_+}$};

\draw (302,309.07) node [anchor=north west][inner sep=0.75pt]  [font=\tiny]  {$q_{I_-}$};

\draw (115,330.07) node [anchor=north west][inner sep=0.75pt]  [font=\tiny]  {$\theta _{+}^{-}$};

\draw (93,374.4) node [anchor=north west][inner sep=0.75pt]  [font=\tiny]  {$\theta _{-}^{-}$};

\draw (241,329.73) node [anchor=north west][inner sep=0.75pt]  [font=\tiny]  {$\theta _{+}^{-}$};

\draw (219,369.07) node [anchor=north west][inner sep=0.75pt]  [font=\tiny]  {$\theta _{+}^{+}$};

\draw (146.30,321.73) node [anchor=north west][inner sep=0.75pt]  [font=\tiny]  {$\pi _{-}$};

\draw (190.33,370.73) node [anchor=north west][inner sep=0.75pt]  [font=\tiny]  {$\pi _{+}$};

\draw (218.33,324.07) node [anchor=north west][inner sep=0.75pt]  [font=\tiny]  {$\mu _{-}$};

\draw (117.67,380.4) node [anchor=north west][inner sep=0.75pt]  [font=\tiny]  {$\mu _{+}$};

\draw (177.33,346.73) node [anchor=north west][inner sep=0.75pt]  [font=\tiny]  {$0$};

\draw (116.67,362.73) node [anchor=north west][inner sep=0.75pt]  [font=\fontsize{0.15em}{0.18em}\selectfont]  {$\bullet $};

\draw (225.33,340.4) node [anchor=north west][inner sep=0.75pt]  [font=\fontsize{0.15em}{0.18em}\selectfont]  {$\bullet $};

\draw (131,333.4) node [anchor=north west][inner sep=0.75pt]  [font=\fontsize{0.71em}{0.25em}\selectfont]  {$\bullet $};

\draw (108,377.73) node [anchor=north west][inner sep=0.75pt]  [font=\fontsize{0.71em}{0.35em}\selectfont]  {$\bullet $};

\draw (229.33,331.07) node [anchor=north west][inner sep=0.75pt]  [font=\fontsize{0.71em}{0.25em}\selectfont]  {$\bullet $};

\draw (204,378.4) node [anchor=north west][inner sep=0.75pt]  [font=\fontsize{0.71em}{0.25em}\selectfont]  {$\bullet $};

\draw (169,355.73) node [anchor=north west][inner sep=0.75pt]  [font=\fontsize{0.71em}{0.25em}\selectfont]  {$\bullet $};

\draw (208,280.4) node [anchor=north west][inner sep=0.75pt]  [font=\fontsize{0.71em}{0.25em}\selectfont]  {$\bullet $};

\draw (88.33,343.07) node [anchor=north west][inner sep=0.75pt]  [font=\fontsize{0.71em}{0.25em}\selectfont]  {$\bullet $};

\draw (290.33,310.4) node [anchor=north west][inner sep=0.75pt]  [font=\fontsize{0.71em}{0.25em}\selectfont]  {$\bullet $};

\end{tikzpicture}
\caption{A drawing of $\Delta$ and $2\Delta$.}
    \label{fig:Parallelo}
\end{figure}

Suppose that $\vartheta$ is a point inside the interior of the line segment $\theta_{+}^+$. Then one verifies that $\vartheta-\sum_{\rho\in I_+}D_\rho\in\theta_+^-$. For any non-empty, proper subset $J\subsetneq I_+$ then $\vartheta-\sum_{\rho\in J}D_\rho$ lies in the interior of $\Delta$. 
Similarly, if $\vartheta\in\theta_-^+$, then $\vartheta+\sum_{\rho\in I_-}D_\rho\in \theta_-^-$ and for any $\emptyset\neq J\subsetneq I_-$, the point $\vartheta+\sum_{\rho\in J}D_\rho$ lies in the interior of the parallelogram $\Delta$. 

By \cite[Corollary 4.8]{BorHor1}, $\dbcoh{\mathcal{X}_\Sigma}$ is generated by line bundles and so it suffices to show that the category $\mathcal{D}$ generated by the line bundles in $\mathcal{S}$ contains all line bundles. Move the polytope $p+\Delta$ by varying $p$ in line with constant $f(p)$ until a new line bundle $\mathcal{L}=\O(E)$ appears. The argumentation for both possible choices of direction is analogous, so we assume that we move in the direction of increasing $\alpha(x)$, i.e. the new line bundle first crosses into the parallelogram $p+\Delta$ via the line segments $p+\theta_\pm^+$.  Note that the image of $\mathcal{L}$ has to lie in either the interior of $p+\theta_+^+$ or the interior of $p+\theta_-^+$ as their intersection point moves along a non-rational line due to the generic choice of $r_i, p$. Both cases are analogous, so we may assume $\mathcal{L}\in p+\theta_+^+$.

We claim that the line bundle $\mathcal{L}$ lies in $\mathcal{D}$. For any line non-empty $J\subsetneq I_+$ the line bundle $\O(E-\sum_{\rho\in J}D_\rho)$ is in $\mathcal{D}$ as its image lies inside the polytope $p+\Delta$ (and the corresponding line bundles are all in $\mathcal{D}$). Consider now the Koszul complex on $\C^{n+3}$ given by $x_\rho$, $\rho\in I_+$. This complex resolves the structure sheaf of a coordinate subspace outside of $U_\Sigma$, yielding a long exact sequence of sheaves on $\mathcal{X}_\Sigma$ (see also \cite{BorHor1}).

Twisting by $\mathcal{L}$, we obtain\[
0\rightarrow \O(E-\sum_{\rho\in I_+}D_\rho)\rightarrow \dots\rightarrow \bigoplus_{\rho\in I_+}\O(E- D_\rho)\rightarrow \mathcal{L}\rightarrow 0.
\]

\noindent Since all terms in this long exact sequence are in $\mathcal{D}$ apart from the last, we also have $\mathcal{L}\in\mathcal{D}$.

Thus, so long as the image $q$ of a line bundle $\mathcal{L}$ satisfies $|f(q-p)|\le \frac{1}{2},$ we have shown $\mathcal{L}\in\mathcal{D}$. A similar argument works for moving the parallelogram in the direction of constant $\alpha$ by using the Koszul resolutions for $I_+$ (or equivalently $I_-$), noting that $-1<f(-\sum_{\rho\in J}D_\rho)<0$ for any non-empty subset $J\subsetneq I_+$. So we have shown that all line bundles lie in $\mathcal{D}$ and thus $\mathcal{S}$ is a full strong exceptional collection.
\end{proof}

Using the explicit full, strong exceptional collection $\mathcal{S}$ constructed above now suffices to prove the Theorem \ref{Thm:AlmSim}.

\begin{proof}[Proof of Theorem \ref{Thm:AlmSim}]
Consider the tilting bundle $\mathcal{T}=\bigoplus_{\mathcal{L}\in \mathcal{S}}\mathcal{L}$ built from the full, strong exceptional collection $\mathcal{S}$. In view of Theorem \ref{Thm:NovaDMstack}, we aim to verify the following vanishing condition:\[
H^i\left(\mathcal{X}_\Sigma, \mathcal{T}\otimes \mathcal{T}^\vee\otimes \O\left(l\cdot\sum_{\rho\in\Sigma(1)}D_\rho\right)\right)=0\text{ for }i,l>0.
\]
The components of $\mathcal{T}\otimes\mathcal{T}^\vee$ are of the form $\mathcal{L}'=\mathcal{L}_1\otimes\mathcal{L}_2^\vee$ for $\mathcal{L}_i\in S$ and so $\mathcal{L}'\in\Delta$. Tensoring with $\O\left(l\cdot\sum_{\rho\in \Sigma(1)}D_\rho\right)$ moves $\mathcal{L}'$ in $\operatorname{Pic}_\R(\mathcal{X}_\Sigma)$ on a line of constant $\alpha$, away from the forbidden cone $F_\emptyset$ and parallel to the supporting hyperplanes of the forbidden cones $F_{I_\pm}$. Thus, for any $l>0$, the line bundle $\mathcal{L}'\otimes \O\left(l\cdot\sum_{\rho\in \Sigma(1)}D_\rho\right)$ remains outside the three forbidden cones and hence is acyclic, as required. Theorem \ref{Thm:NovaDMstack} thus shows that the toric algebra associated to $\cone(Q\times\{1\})$ has a toric NCCR. Since $P\times\{1\}$ is (up to lattice equivalence) a face of $Q$, Theorem \ref{Thm:FaceOfReflexiveTORICNCCR} shows that $R=k[\sigma^\vee\cap M]$ has a toric NCCR, as claimed.
\end{proof}

\begin{remark}
    Whilst Corollary \ref{Cor:SimplicialHasNCCR} and Theorem \ref{Thm:AlmSim} are not new results, having been proven in \cite{FMS19} and \cite{Tomo25} respectively, the descent of toric NCCRs via Theorem \ref{Thm:FaceOfReflexiveTORICNCCR}, we provides a unified approach that works in both cases. 
    Affine toric Gorenstein varieties of higher Picard rank no longer necessarily admit toric NCCRs, as a counterexample by  \v{S}penko and Van den Bergh \cite[Example 6.1]{SVdBtoricI} shows. The toric algebra is associated to a four-dimensional cone $\sigma=\cone(P\times\{1\})$ with 6 extremal rays, for which no toric NCCRs exist. Consequentially, any reflexive polytope containing the polytope $P$ as face itself cannot admit a toric NCCR. Thus, any smooth toric DM stack $\mathcal{X}_\Sigma$ where $\Sigma$ simplicially subdivides the face fan of $P$ cannot admit a tilting bundle that fulfills the required cohomology vanishing condition to give a tilting bundle on $[\tot\omega_{X_\Sigma}]$. However, the authors hope that proving Conjecture \ref{Conj:FaceRkBigger1} can help give a general proof to Conjecture \ref{Conj:affinetoric} in a similar manner as was done here for the cases of (almost) simplicial toric algebras.
\end{remark}

\bibliographystyle{alpha}
\bibliography{bib}
\end{document}